\documentclass[a4paper]{amsproc}
\usepackage{amsmath}
\numberwithin{equation}{section}
\usepackage{amsthm}
\usepackage{amsfonts}
\usepackage{amssymb}
\usepackage{mathrsfs}
\usepackage{tikz}
\usepackage{environ}
\usepackage{elocalloc}
\usepackage{url}
\usepackage{caption}
\usepackage{CJKutf8}
\usepackage{caption}
\usepackage{mathtools}
\usepackage{dcpic}
\usepackage{tikz-cd}
\usepackage{hyperref}
\usepackage{enumitem}
\usepackage{stmaryrd}
\usetikzlibrary{positioning}
\usetikzlibrary{arrows,chains,positioning,scopes,quotes}
\usetikzlibrary{decorations.markings}

\title{Stationary one-sided area-minimizing hypersurfaces with isolated singularities}
\author{Zhenhua Liu}
\dedicatory{Dedicated to Xunjing Wei}
\begin{document}
	\maketitle
	\begin{abstract}
		We extend the results of Hardt and Simon in \cite{HS} to prove that isolated singularities of stationary one-sided area-minimizing hypersurfaces can be locally perturbed away on the side that they are minimizing.
	\end{abstract}

	\newcommand{\ai}{\alpha}
	\newcommand{\su}{\Subset}
	\newcommand{\be}{\beta}
	\newcommand{\Ga}{\Gamma}
	\newcommand{\ga}{\gamma}	
	\newcommand{\de}{\delta}
	\newcommand{\De}{\Delta}
	\newcommand{\e}{\epsilon}
	\newcommand{\lam}{\lambda}
	\newcommand{\di}{\textnormal{div}}
	\newcommand{\Lam}{\Lamda}
	\newcommand{\om}{\omega}
	\newcommand{\Om}{\Omega}
	\newcommand{\si}{\sigma}
	\newcommand{\Si}{\Sigma}
	\newcommand{\vp}{\varphi}
	\newcommand{\rh}{\rho}
	\newcommand{\ta}{\theta}
	\newcommand{\spt}{\textnormal{spt}}
	\newcommand{\Ta}{\Theta}
	\newcommand{\W}{\mathcal{O}}
	\newcommand{\ps}{\psi}
	
	\newcommand{\mf}[1]{\mathfrak{#1}}
	\newcommand{\ms}[1]{\mathscr{#1}}
	\newcommand{\mb}[1]{\mathbb{#1}}
	\newcommand{\cd}{\cdots}
	
	\newcommand{\s}{\subset}
	\newcommand{\es}{\varnothing}
	\newcommand{\cp}{^\complement}
	\newcommand{\bu}{\bigcup}
	\newcommand{\ba}{\bigcap}
	\newcommand{\co}{^\circ}
	\newcommand{\hm}{\mathcal{H}}
	\newcommand{\ito}{\uparrow}
	\newcommand{\wc}{\rightharpoonup}
	\newcommand{\dto}{\downarrow}

	\newcommand{\ti}[1]{\tilde{#1}}
	\newcommand{\la}{\langle}
	\newcommand{\ra}{\rangle}
	\newcommand{\ov}[1]{\overline{#1}}
	\newcommand{\no}[1]{\left\lVert#1\right\rVert}
	\DeclarePairedDelimiter{\cl}{\lceil}{\rceil}
	\DeclarePairedDelimiter{\fl}{\lfloor}{\rfloor}
	\DeclarePairedDelimiter{\ri}{\la}{\ra}

	\newcommand{\du}{^\ast}
	\newcommand{\sm}{\setminus}
	\newcommand{\pf}{_\ast}
	\newcommand{\is}{\cong}
	\newcommand{\n}{\lhd}
	\newcommand{\m}{^{-1}}
	\newcommand{\reg}{\textnormal{reg}}
	\newcommand{\ts}{\otimes}
	\newcommand{\ip}{\cdot}
	\newcommand{\op}{\oplus}
	\newcommand{\xr}{\xrightarrow}
	\newcommand{\xla}{\xleftarrow}
	\newcommand{\xhl}{\xhookleftarrow}
	\newcommand{\xhr}{\xhookrightarrow}
	\newcommand{\mi}{\mathfrak{m}}
	\newcommand{\lp}{\llcorner}
	\newcommand{\wi}{\widehat}
	\newcommand{\ma}{\mathbb{M}}
	\newcommand{\sch}{\mathcal{S}}

	\newcommand{\w}{\wedge}
	\newcommand{\X}{\mathfrak{X}}
	\newcommand{\pd}{\partial}
	\newcommand{\dx}{\dot{x}}
	\newcommand{\dr}{\dot{r}}
	\newcommand{\dy}{\dot{y}}
	\newcommand{\dth}{\dot{theta}}
	\newcommand{\pa}[2]{\frac{\pd #1}{\pd #2}}
	\newcommand{\na}{\nabla}
	\newcommand{\dt}[1]{\frac{d#1}{d t}\bigg|_{ t=0}}
	\newcommand{\ld}{\mathcal{L}}

	\newcommand{\N}{\mathbb{N}}
	\newcommand{\R}{\mathbb{R}}
	\newcommand{\Z}{\mathbb{Z}}
	\newcommand{\Q}{\mathbb{Q}}
	\newcommand{\C}{\mathbb{C}}
	\newcommand{\bh}{\mathbb{H}}
	
	\newcommand{\lix}{\lim_{x\to\infty}}
	\newcommand{\li}{\lim_{n\to\infty}}
	\newcommand{\infti}{\sum_{i=1}^{\infty}}
	\newcommand{\inftj}{\sum_{j=1}^{\infty}}
	\newcommand{\inftn}{\sum_{n=1}^{\infty}}	
	\newcommand{\snz}{\sum_{n=-\infty}^{\infty}}	
	\newcommand{\ie}{\int_E}
	\newcommand{\ir}{\int_R}
	\newcommand{\ii}{\int_0^1}
	\newcommand{\sni}{\sum_{n=0}^\infty}
	\newcommand{\ig}{\int_{\ga}}
	\newcommand{\pj}{\mb{P}}
	
	\newcommand{\io}{\textnormal{ i.o.}}
	\newcommand{\aut}{\textnormal{Aut}}
	\newcommand{\out}{\textnormal{Out}}
	\newcommand{\inn}{\textnormal{Inn}}
	\newcommand{\mult}{\textnormal{mult}}
	\newcommand{\ord}{\textnormal{ord}}
	\newcommand{\F}{\mathcal{F}}
	\newcommand{\V}{\mathbf{V}}	
	\newcommand{\II}{\mathbf{I}}
	\newcommand{\ric}{\textnormal{Ric}}
	\newcommand{\sef}{\textnormal{II}}
	
	\newcommand{\wh}{\Rightarrow}
	\newcommand{\eq}{\Leftrightarrow}
	
	\newcommand{\eqz}{\setcounter{equation}{0}}
	\newcommand{\se}{\subsection*}
	\newcommand{\ho}{\textnormal{Hom}}
	\newcommand{\ds}{\displaystyle}
	\newcommand{\tr}{\textnormal{tr}}
	\newcommand{\id}{\textnormal{id}}
	\newcommand{\im}{\textnormal{im}}
	\newcommand{\ev}{\textnormal{ev}}
	
	\newcommand{\gl}{\mf{gl}}
	\newcommand{\sll}{\mf{sl}}
	\newcommand{\cs}{\Subset}
	\newcommand{\so}{\mf{so}}
	\newcommand{\ad}{\textnormal{ad}}

	\theoremstyle{plain}
	\newtheorem{thm}{Theorem}[section]
	\newtheorem{lem}[thm]{Lemma}
	\newtheorem{prop}[thm]{Proposition}
	\newtheorem*{remark}{Remark}
	\newtheorem*{cor}{Corollary}
	\newtheorem*{pro}{Proposition}
	
	\theoremstyle{definition}
	\newtheorem{defn}{Definition}[section]
	\newtheorem{conj}{Conjecture}[section]
	\newtheorem{exmp}{Example}[section]
	
	\theoremstyle{remark}
	\newtheorem*{rem}{Remark}
	\newtheorem*{note}{Note}
	\section{Introduction}
	Area-minimizing hypersurfaces in $n+1$-dimensional manifolds $M^{n+1}$ are known to be smooth outside a set of Hausdorff dimension at most $n-7$ (\cite{HF2}). However, little is known about the geometry of the singular sets. In this regard, the best results to date are the structure theory of area-minimizing hypersurfaces with isolated singularities developed by Robert Hardt and Leon Simon in \cite{LS2} and \cite{HS}. Roughly speaking, they have proved that
	\begin{itemize}
		\item[(i)] On any side of an area-minimizing hypercone, all area-minimizing boundaries confined in that side are smooth and unique up to scalings.
		\item[(ii)] (Based on (i)) On any side of an area-minimizing hypersurface, if the $n-1$ boundary of a Plateau problem is close enough to a smooth $n-1$ submanifold of the cone, then the solution to the Plateau problem is smooth.
	\end{itemize}
	We will explain the terminologies contained in those statements later. These statements basically say that isolated singularities of area-minimizing hypersurfaces behave as well as one can imagine, in that one can always perturb them away locally in a suitable sense to get smooth objects. 
	
	In this paper, we extend these results to hypersurfaces that are only area-minimizing on one side. In other words, the area of such a hypersurface will not decrease if we deform it to one side of the original hypersurface. Most of the statements are simple adaptations. However, the hardest part to deal with is the convergence of varifold compared with the convergence of current for those one-sided minimizing objects. Since they're not even locally minimizing, a general compactness theorem like the one for minimizing currents is impossible. Thus, we no longer have simultaneous convergence to the same thing modulo orientation in the two different topologies of currents and varifolds, and indeed there are simple counterexamples. We deal with this by exploiting special structures of one-sided area minimizing currents.
	\section{Preliminaries}
	In our paper, we deal with integral currents, sets of finite perimeter (Caccioppoli sets) and integral varifolds. Basically, an integral $k$ current $T$ in an $n+1$-dimensional manifold is a $k$-rectifiable set $R$ in $M$ equipped with a $\hm^{k}$ measurable and integrable integer-valued function $\ta$ and a $\hm^k$-measurable and integrable $k$-vector-valued function $\nu$. It acts on $k$-forms on $M$ by
	\begin{align*}
	T(\om)=\int_R \ri{\om,\nu}\ta d\hm^k.
	\end{align*} We denote the rectifiable set $R$ as $\text{spt}T,$ and $\ta$ the multiplicity of $T.$ We will use $\ma(T)$ to denote the mass of $T.$ The associated integral varifold $V=\mu_T$ of $T$ is defined as dropping the orientation, and thus a measure only (Chapter 4 of \cite{LS1}). A set of finite perimeter in an open set $\Om$ is a Borel set $F$ whose perimeter in $\Om$ defined as
	\begin{align*}
	P(F,\Om)=\sup\left\{\int_{\Om}1_F\di_M Xd\mu_{M}\bigg||X|\le 1,X\in C^1_c(M,TM)\right\},
	\end{align*} is finite. In other words, the indicator function $1_F$ is a BV function in $\Om.$ For introductions to these topics, please refer to \cite{HF1} for currents, \cite{WA1} for varifolds, \cite{LA}, \cite{FM}, and \cite{MM} for sets of finite perimeters and \cite{LS1} for all of them.
	\begin{defn}\label{defoneside}
		
		(Compare Definition in Section II of \cite{LF3}) Let $E$ be an open set which has finite perimeter in $M^{n+1}$. Suppose $T=\pd [[E]]$ is the boundary and $T$ is stationary as a current. We say $T$ is a one-sided area-minimizing hypersurface if for any $n$-dimensional integral current $X$ with $\pd X=0,$ and $\spt X\s \ov{E},$ we have
		\begin{align*}
		\ma(T)\le \ma(T+X).
		\end{align*}
	\end{defn}
	Why do we impose the stationarity? Using elementary geometry, one can show that any convex polygon in $\R^2$ is one-sided area (length) minimizing in the unbounded component of $\R^n$ that it divides into. Thus, in order to deduce any minimal surface type codimension $7$ regularity, we have to assume that $T$ is stationary.
	\begin{prop}\label{regone}
		For any one-sided area minimizing current $T=\pd[[E]]$ as defined above, the support of $T$ is smooth outside a set of Hausdorff dimension at most $n-7$, (discrete when $n-7$), and the associated varifold $\mu_T$ is stable. 
	\end{prop}
	\begin{proof}
		As a conclusion of \cite{LF3}, any one-sided area-minimizing current is smooth except for a singular set of Hausdorff dimension at most $n-7.$ 
		
		Alternatively, we can prove this by invoking the regularity of stable varifolds as in \cite{NW1} by ruling out codimension $1$ singularities. The same reasoning as the last part of this proof shows that $\mu_T$ is stable in its regular part, thus by the Regularity and Compactness Theorem in \cite{NW1}, if we can rule out singularities which comes from transverse intersections of $C^{1,\ai}$ manifolds, then we get the usual codimension $7$ regularity. By focusing on a very small geodesic ball, such intersections roughly look like transversely intersecting half-hyperplanes. Thus, we can cut always short the area by rounding the corners. For the details, a straightforward adaptation of ruling out codimension $1$ singularity in Theorem 7.2 in \cite{BW} can be used.
		
		Next, we prove that $\mu_T$ is stable. By Remark 27.7 of \cite{LS1}, we have $\ma_W(T)=\no{\mu_T}(W)$ for any open set $W.$ For any one-parameter families of diffeomorphisms $\phi_t$ generated by a vector field $X.$ We always have $\ma_W((\phi_t)\pf T)=((\phi_t)\pf \mu_T)(W)$ (see comments just before Theorem 27.3 in \cite{LS1}). Thus, the stationarity of $T$ as a current and $\mu_T$ as a varifold are the same. Since $T$ is smooth outside a set of Hausdorff dimension at most $n-7,$ the stability of the varifold is equivalent to stability in the regular part. (Vanishing $n-2$-Hausdorff dimension of singular parts suffices for this statement by using cutting-off functions on the regular part.) For any vector fields $X$ supported on the regular part of $T$, we use $\phi_t$ to denote the associated 1-parameter family of diffeomorphisms. If $X$ points into $E$, then we have
		\begin{align*}
		\frac{d^2}{dt^2}\hm^n(\phi_t(\reg{T}))\ge 0
		\end{align*}by the definition of being one-sided area-minimizing. For general $X,$ we can decompose $X$  out of $X=U-V,$ with $U,V$ pointing out of $E$ and $U,V$ orthogonal (vanishing when the other is not). $U,V$ are only Lipschitz continuous but the second variation formula in manifolds for stationary hypersurfaces (such as 1.143 in \cite{CM}) still holds. We have
		\begin{align*}
		\frac{d^2}{dt^2}\hm^n(\phi_t(\reg{T}))=&-\int_{\reg{T}}\no{\ri{A(\ip,\ip),X}}^2+\int_{\reg{T}}|\na^N_{\reg{T}} X|^2-\int_{\reg{T}}\tr_{\reg{T}}\ri{R_M(\ip,X)\ip,X}\\
		=&-\int_{\reg{T}}\no{\ri{A(\ip,\ip),U}}^2+\int_{\reg{T}}|\na^N_{\reg{T}} U|^2-\int_{\reg{T}}\tr_{\reg{T}}\ri{R_M(\ip,U)\ip,U}\\
		&-\int_{\reg{T}}\no{\ri{A(\ip,\ip),-V}}^2+\int_{\reg{T}}|\na^N_{\reg{T}} (-V)|^2-\int_{\reg{T}}\tr_{\reg{T}}\ri{R_M(\ip,-V)\ip,-V}\\
		=&-\int_{\reg{T}}\no{\ri{A(\ip,\ip),U}}^2+\int_{\reg{T}}|\na^N_{\reg{T}} U|^2-\int_{\reg{T}}\tr_{\reg{T}}\ri{R_M(\ip,U)\ip,U}\\
		&-\int_{\reg{T}}\no{\ri{A(\ip,\ip),V}}^2+\int_{\reg{T}}|\na^N_{\reg{T}} V|^2-\int_{\reg{T}}\tr_{\reg{T}}\ri{R_M(\ip,V)\ip,V}\\
		\ge& 0,
		\end{align*}
		because the last two lines corresponds to the sum of the second variation with respect to $U$ and $V$.
	\end{proof}
	The following lemma corresponds to Lemma 1.16 in \cite{HS}, and the proof below is an adaptation of the proof of that.
	\begin{lem}\label{l1.16}
		Suppose $T=\pd[[E]]$ is stationary one-sided minimizing in $U$ with only isolated singularities. If $W$ is a $C^2$ domain, $V$ is open and we have
		\begin{align*}
		W\cs V\cs U, \spt T\cap \pd W\s \reg T,
		\end{align*} with the intersection being transverse. Then there exists $\e=\e(U,W,V,T)>0$ so that for $\Ga_0=\pd(T\llcorner W ),\Ga=\vp\pf \Ga_0$ with 
		\begin{align*}
		\vp\in C^2(\Ga_0,U),\vp(\Ga_0)\s E\cap \pd W,
		\end{align*}and $\no{\vp-i_{\Ga_0}}_{C^2}<\e,$ \footnote{The $C^2$ norm is obtained by embedding $M$ isometrically out of some Euclidean space and then using the $C^2$ norm induced from the Euclidean space} we have an $S$, minimizing in $E\cap V$, with $\pd S=\Gamma,$ $\spt S\s \ov{W\cap {E}},$ $\spt S\cap \pd W=\Ga$ and $S=\pd[[F]]\lp W$ for some open $F\s W$ with $\pd F\cap W=\spt S\cap W.$
	\end{lem}
	\begin{proof}
		Pick any $\Ga_j=(\vp_j)\pf \Ga_0,$ with $\vp_j\in C^2$, $\Ga_j\s E\cap \pd W,$ $\no{\vp_j-i_{\Ga_0}}_{C^2}\le \frac{1}{j}$. Let $S_j$ be the integral current that solves the Plateau problem with boundary $\Ga_j$ and competitors in $\ov{V\cap E}$. In other words,
		\begin{align*}
		M(S_j)=\inf\{M(S)|S\text{ is integral current supported in }\ov{E\cap V},\pd S=\Ga_j\}
		\end{align*}
		Such a minimizer exists by compactness of currents \cite{HF1}. Moreover, it's a stationary current by Lemma 1.20 in \cite{HS}. Recall that area formula depends only on the first derivative of the maps. Thus, we have $$\ma(S_j)\le\ma((\phi_j)\pf T\lp W)\le \ma(T\lp W)+\frac{c}{j}$$ by construction. By compactness theorem of integral currents, we can extract a subsequence (not relabeled) $S_j\to S$ in flat norm. By lower semi-continuity of mass, we have $\ma(S)\le\ma(T\lp W).$ Since $T$ is one-sided area minimizing, we have $\ma(T\lp W)\le \ma(S),$ and thus $\ma(T\lp W)=\ma(S).$ Replicating the argument for compactness of minimizing hypercurrents as in \cite{LS1}, we can deduce that $S$ is a solution of Plateau problem with boundary $\Ga$ and competitors contained in $\ov{E}.$ 
		
		Consider the cycle $T'=T-T\lp W+S$, i.e., replacing a portion of $T$ with $S.$ Note that we have $\ma(T')\le\ma (T'+X)$ for any integral current $X$ with $\pd X=0$ and $\spt X\cs \ov{E\cap V}$, since $\ma(T'+X)=\ma(T+(W-T\lp W+X))$ and $\pd(W-T\lp W+X)=0.$ This implies $T'$ is area-minimizing in $\ov{E\cap V}$. Since we have $\pd T'=0,$ by the decomposition of hypercurrents (Corollary 27.8 in \cite{LS1}) and a straightforward adaptation of Lemma 33.4 in \cite{LS1}, we deduce that for any point $p\in \Ga$, there exists a neighborhood $B_r(p)$ so that $T'\lp B_r(p)=\sum_j \pd[[U_j]]$, where $U_j$ are sets of finite perimeter and $\pd[[U_j]]$ is area-minimizing in $\ov{E\cap V}$, and thus stationary one-sided area-minimizing. By Lemma 1.20 in \cite{HS}, each $\pd[[U_j]]$ is stationary. We can invoke the regularity of stationary one-sided area-minimizing currents (Proposition \ref{regone}) to deduce that $T'$ is smooth outside a set of Hausdorff dimension at most $n-7.$ By the strong maximum principle (Theorem 1.1 in \cite{NW2}) for stationary varifolds, we deduce that $\spt T'=\spt T.$ Since $\ma(T')=\ma(T)$ we deduce immediately that $T'=T.$ This shows $S=T\lp W.$ 
		By Lemma 1.20 in \cite{HS}, we see that $S_j$ are all stationary currents. Thus, by the boundary regularity of $S=T\lp W$ and Allard's boundary regularity theorem \cite{WA2}, we deduce that $\{S_j\}$ are smooth manifolds with boundary on a neighborhood of $\pd S.$ 
		
		The rest of the proof is the same as the proof of Lemma 1.17 starting with its formula (1) in \cite{HS}.
	\end{proof}
	
	\section{Unique foliation on the minimizing side of a hypercone}
	Suppose $C=\pd [[E]]$ is a stationary one-sided area-minimizing hypercone in $\R^{n+1}$ with an isolated singularity at the origin. By maximum principle, $C\cap S^n$ has only one connected component (otherwise some rotation gives the contradiction). Thus, $C$ separates $\R^{n+1}$ into two open connected sets.
	
	We define $\eta_{y,\lam}(x)=\lam\m(x-y)$.
	\begin{thm}\label{t2.1}
		There exists an oriented connected embedded real analytic hypersurface $S\s E$ with $S=\pd[[F]],$ $\ov{F}\s E $, $F$ open, $\textnormal{sing}S=\es,$ $\text{dist}(S,0)=1$ with the following properties
		\begin{itemize}
			\item[(i)] $S$ is area-minimizing in $\ov{E}$;
			\item[(ii)] for any $y\in E$, the ray $\{ty:t>0\}$ intersects $S$ at a single point, and the intersection is transverse;
			\item[(iii)] if $T$ is any other multiplicity one minimizing integral current with $T=\pd [[G]]$ for some $G\s \ov{E}$ and area-minimizing in $\ov{E}$, then either $T=C$ or $T=(\eta_{0,\lam})\pf S$ for some $\lam>0.$
			\item[(iv)] if $V$ is any other stationary boundary of Caccioppoli set $V=\pd [[H]]$ for some $H\s \ov{E}$ and one-sided area-minimizing in $H$, then either $T=C$ or $T=(\eta_{0,\lam})\pf S$ for some $\lam>0.$
		\end{itemize}
	\end{thm}
	\begin{proof}
		The basic idea is to extract a subsequence from a suitably blowing-up sequence of currents minimizing in $\ov{E}.$  
		
		For (i)-(iii), a straightforward adaptation of the proof of Theorem 2.1 in \cite{HS} applies here. We only need to substitute Lemma 1.16 with our Lemma \ref{l1.16} above. 
		
		For (iv) we need to do a little more work. By Proposition \ref{regone}, we know that $\mu_V$ is a stable varifold. By comparison with Euclidean spheres and one-sided minimizing, for any $p$ in the support of $\mu_V,$ we can deduce that $\mu_V(B_r(p))\le \om_n r^n,$ where $\om_n$ is the volume of $n$-sphere. Thus, by the compactness theorem of varifolds \cite{WA1} and the regularity theorem in \cite{NW1}, we deduce that for any sequence of $\rh_j\to\infty, $ there exists a not relabeled subsequence so that the blow-down $(\rh_j)\du\mu_V\to \mu_\rh$ for some stable varifold $\mu_\rh$ contained in $\ov{E},$ with singular set of Hausdorff dimension at most $n-7.$ Moreover, note that by Hausdorff convergence of support of converging stationary varifolds in compact sets, we must have $0\in\spt{\mu_\rh}.$ By Theorem 1.1 (maximum principle) of \cite{NW2}, we deduce that the support of $\mu_\rh$ coincides with the support of $\pd[[E]].$ By taking a further subsequence we can assume that $(\rh_j\m)\pf V$ converges at the same time as a current to some $V_\rh'=\pd[[H_\rh]],$  with $ H_\rh\s E$ (since $1_E\ge 1_{H_\rh}$ a.e.) By Theorem 1.2 (convergence of associated varifold with added terms) in \cite{BW2}, the support of $V'_\rh$ is contained the support of $\mu_{\rh}$. Since $V'_\rh=\pd[[H_\rh]]$ is a boundary, this can happen only if $H_\rh=\es$ or $H_\rh=E.$ $H_\rh=\es$ cannot happen because the support of $(\rh_j\m)\pf V$ is the same as the support of $(\rh_j)\du \mu_V$, and thus also converges in Hausdorff distance to the support of $\pd [[E]]$, which implies that far off from $\pd[[E]]$, $1_{(\rh_j\m)\pf H}$ is always nonzero.
		
		Thus, we have $(\rh_j\m)\pf H\wc E.$ If we can show that the associated varifold also converges to $\mu_{\pd[[E]]}$, then the same reasoning as in (iii) gives the desired result (using Allard regularity to express the blow-down as graphs and then use \cite{CHF} as in \cite{HS}). To prove this, it suffices to prove that $\ma_{B_r(0)}(\pd[[E]])=\lim_{j\to\infty} \ma_{B_r(0)}((\rh_j\m)\pf \pd[[H]])$ for any $r$, because we already know that $\mu_{\pd[[E]]}=\mu_{\pd[[H_\rh]]}+2W$ for some integral varifold by Thoerem 1.2 of \cite{BW2}. The basic idea is to construct almost conical surfaces $C_j$ in $B_r(0)$ lying on the minimizing side of $(\rh_j\m)\pf \pd[[H]]$ that have area very close to $\pd[[E]]\lp B_r(0).$ Thus, we can deduce inequalities of the form \begin{align}\label{ccp}
		\ma_{B_r(0)}((\rh_j\m)\pf V)\le \ma_{B_r(0)}C_j\le \ma_{B_r(0)}\pd[[E]]+ a_j,
		\end{align} with $a_j\to 0.$ This yields that $\ma_{B_r(0)}\pd[[E]]\ge \limsup_j\ma_{B_r(0)}((\rh_j\m)\pf V).$ Combined with the lowersemi continuity of mass, this gives the desired convergence
		
		The almost conical comparison surfaces can be constructed as follows. For any fixed radius $r$, the support of $\mu_{(\rh_j\m)\pf V}$ is within Hausdorff distance strictly smaller than $\e_j$ from $\pd[[E]]\lp B_r(0)$, with $\e_j\to 0.$ Let $\de$ denote a fixed radius of a smooth tubular neighborhood of the link of $\pd[[E]]$ and $S^n$ in $S^n.$ (We can simply impose it to be a little smaller than the largest possible such radius). For $j$ large enough, consider the tubular neighborhood $U_j$ of length $\e_j$ of $\pd[[E]]$ in $E\cap B_r(0)\sm B_{\e_j/\de}(0).$ Then $P_j=\pd U_j\sm\pd E$ is a smooth hypersurface with boundary by construction, with area $O(\e_j)$ close to that of $\pd[[E]]\cap B_r(0)\sm B_{\e_j/\de}(0)$. Moreover, $P_j$ lies in $(\rh_j\m)\pf H$ since it's more than $\e_j$ away from $\pd[[E]]\lp B_r(0)$. For almost every radius $r,$ $(\rh_j\m)\pf H$ intersect $\pd B_r(0)$ transversely for every $j$ (for each $j$ this is apparently true, and we only have countably many $j$), so without loss of generality, we can suppose the intersection is transverse for the $r$ of our choice. Then we can adjoint with appropriate oreintation the region on $E\cap B_{\e_j/\de}(0)$ bounded by $P_j,$ (which has area of $O(\e_j^n)$) and on $E\cap B_r(0)$ bounded by $P_j$ and $(\rh_j\m)\pf \pd[[H]]$ (which has area of $O(\e_j)$) to obtain a current $\ti{P}_j$ with $\pd\ti{P}_j=\pd ((\rh_j\m)\pf V)\lp B_r(0)$, and $\ma_{B_r(0)}(\ti{P_j})\le \ma_{B_r(0)}(E)+O(\e_j).$ Since $(\rh_j\m)\pf \pd[[H]]$ is one-sided minimizing, we deduce that $\ma_{B_r(0)}((\rh_j\m)\pf \pd[[H]])\le \ma_{B_r(0)}(\ti{P_j}),$ which yields inequality (\ref{ccp}) as claimed.
	\end{proof}

	\section{Local perturbations}
	Suppose $T=\pd[[E]]$ is a stationary one-sided area-minimizing boundary with isolated singularities and $E$ an open Caccioppoli set. Moreover, for every point of $T,$ there exists a multiplicity one varifold tangent cone with isolated singularity. (Multiplicity one refers to the multiplicity at smooth points of the tangent cone, and by Corollary of Theorem 5 of \cite{LS2}, such tangent cones are unique.) Suppose $U,V$ are open sets with $E\s U$ and $V\cs U$, with $\spt T\cap U=\pd E\cap U$. Moreover, suppose we have a $C^2$ domain $W\cs V$ so that the intersection $\spt T\cap \pd W\s \reg T$ is transverse. Then we have the following.
	\begin{thm}\label{t5.1}
		There is an $\e=\e(E,U,V,W)>0$, so that if $\Ga=\vp\pf \Ga_0$ for some $\vp\in C^2(\Ga_0,\pd W)$ with $\Ga\s \pd W\cap E$, $|\vp-i_{\Ga_0}|_{C^2}<\e,$ and each component of $\Ga$ intersects $E$ nontrivially, then $\Ga$ bounds an integer multiplicity current $S$ with $S$ minimizing in $V$, and $\textnormal{sing}S=\es$ for any such $S.$
	\end{thm}
	\begin{proof}
		The same proof of Theorem 5.6 in \cite{HS} applies. We only have to substitute Lemma 1.16 in \cite{HS} with our Lemma \ref{l1.16} and Theorem 2.1 in \cite{HS} with our Theorem \ref{t2.1}.
	\end{proof}
	Next, we will strengthen Theorem \ref{t5.1} in a special case. 
	\begin{lem}\label{utc}
		If a varifold tangent cone of a one-sided minimizing $T$ at some point $p$ only has an isolated singularity, then the varifold tangent cone is unique at that point $p$ and coincides with the varifold associated to unique tangent current at $p$.
	\end{lem}

	\begin{proof}
		The idea of the proof is very similar to that of the Theorem \ref{t2.1} (iv). First, since a varifold tangent $V_p$ cone of $T$ at $p$ has only isolated singularity, its link $L$ with $S^n$ would be a connected smooth hypersurface. The link $L$ is connected by maximum principle, and thus is a boundary of a connected region $O$ by basic algebraic topology. ($V_p$ is indeed an integer multiple of the cone over $L$ as a varifold by the Regularity and Compactness Theorem in \cite{NW1}.) Recall that $V_p=(\eta_{p,\rh_j})\du \mu_T$ for some some sequence $\rh_j\to\infty$ by definition. By taking a not relabeled subsequence, we can assume that $(\eta_{p,\rh_j\m})\pf T$ converges to a tangent current $T_p=\pd[[E_p]]$ for some open set $E_p.$ Again, the support of $T_p$ is contained in $\spt (V_p)$ by Theorem 1.2 in \cite{BW2}. By intersecting $E_p$ with $S^n,$ we can get an open set with boundary contained in $L$. Thus, either we have $E_p\cap S^n=\es$ or  $S^n,$ or $E_p\cap S^n$ must coincide with $O$ if we choose $O$ accordingly. By Hausdorff convergence of support of the associated varifold to the blow-up sequence of $T,$ when $\rh_j$ is large enough, the associated $(\eta_{p,\rh_j\m})\pf E$ has boundary $\e$ close to $\spt T$, thus, it always contains an open subset $\e$ away from $\spt T$ of nontrivial $\hm^{n}$ measure. This rules out the case $E_p=\es$. The case for $E_p=\R^n$ can be similarly ruled out. Thus we must have $E_p$ is just the cone over $O$ suitably oriented. Then, the same almost conical comparison argument as in the proof of the Theorem \ref{t2.1} (iv) can be applied to deduce that $V_p$ is just the varifold associated with $T_p$. Since $\mu_{T_p}$ is of multiplicity one, we can use Corollary of Theorem 5 in \cite{LS2} to deduce the uniqueness of tangent varifold, and thus the uniqueness of tangent current.
	\end{proof}
	\begin{remark}
		The condition in Lemma \ref{utc} is automatically satisfied if $n=7$ by Proposition \ref{regone}.
	\end{remark}
	We will use this proposition to prove a strengthened version of Theorem \ref{t5.1} in dimension $n+1=8.$
	\begin{thm}\label{osp}
		Assume the same conditions as in Theorem \ref{t5.1}, and $n+1=8.$ Suppose we have a sequences of stationary one-sided area minimizing hpyersurfaces $T_j=\pd[[E_j]]$ lying on one-side of $T,$ i.e., $E_j\s E,$ so that the boundaries $\pd(T_j\lp W)$ in $W$ lie exclusively on $\pd W,$ in other words $\spt\pd(T_j\lp \ov{W})\s \pd W.$ Moreover, the boundary $\pd(T_j\lp W)$ never coincides with $\pd(T\lp W)$. If $T_j\to T$ as current, and $\mu_{T_j}\to \mu_T$ as varifold simultaneously, then for $j$ large enough, $T_j$ is smooth everywhere.
	\end{thm}
	\begin{proof}
		The proof is in a spirit similar to Theorem 5.6 in \cite{HS}. However, there is something tricky about the convergence of varifolds versus that of the currents. The dimension assumption is used in two ways. The first is to use Lemma \ref{utc}, and the second is to deduce that the supports of some stable varifolds involved are boundaries Caccioppoli sets.
		
		By the transversality assumption, there is no boundary singularity. Since the singular set is discrete, there are only finitely many singular points. Pick an interior singular point $p$ of $T\lp W$. Without loss of generality, we can assume this is the only singular point, because proving we can reiterate the proof to arrive at the desired conclusion for the case of several singular points.  By Allard boundary regularity theorem (\cite{WA1}) and the varifold convergence $\mu_{T_j}\to\mu_T,$ we deduce that for $j$ large enough, $T_j$ can only have interior singularities, and they stay away from an $\e$-neighborhood of $\pd T.$ Without loss of generality, we can suppose every $T_j$ has at least one interior singularity $s_j.$ (If only finitely many $T_j$ is singular, then there is nothing to prove. If there are infinitely many, we can just do the reasoning for them.) By the previous reasoning of boundary regularity and Allard interior regularity theorem (\cite{WA2}), we can only have $s_j\to p.$ By maximum principle in \cite{NW2}, we know that $|s_j-p|>0$ for all $j$ since $\pd (T_j\lp W)$ never coincides with $\pd(T\lp W).$ Take the normal coordinate at $p.$ We can do a blow-up and obtain a not-relabeled subsequence so that  the blow-up currents $(|s_j-p|\m)\pf T$ converge to the tangent current $T_p=\pd[[E_p]]$ and the associate varifold $(|s_j-p|)\du \mu_{T}$ converges to its tangent varifold $V_p$. By Proposition \ref{utc}, the two tangents are unique and $V_p=\mu_{T_p}.$ By taking further subsequences, we can assume that $(|s_j-p|\m)\pf T_j$ converges to some boundary of Caccioppoli set $Z=\pd[[O]]$, and the associated varifolds $(|s_j-p|)\du \mu_{T_j}$ converges to some $\ti{V}.$ A prior, we don't know that $\ti{V}=\mu_{Z}$. However, this is indeed the case. We can argue as follows. By the Regularity and Compactness Theorem in \cite{NW1}, we know that $\spt \ti{V}$ is a hypersurface with only discrete singularities. Moreover, the singularities are conical. This implies that each $\spt \ti{V}$ is triangulable and is a cycle as a simplicial complex. Thus, by basic algebraic topology, we know that $\spt \ti{V}=\pd A$ for some Caccioppoli sets $A.$ Again, since each blow-up $Z$ has support $\spt Z\s \spt \ti{V}$, we're left with either $O=A$ (by choosing $A$ appropriately) or $O=\es $ or $O=\R^{n+1}.$ The latter two cases can be ruled out similarly as in the proof of Proposition \ref{utc}. We're left with $O=A$ and we can use a tubular strip comparison argument as in the proof of Theorem \ref{t2.1} (iv) to deduce that we must have $\ti{V}=\mu_{Z}.$
		
		Now, if we can prove $Z$ is stationary one-sided area-minimizing, then by Theorem \ref{t2.1} (iv), either we have $Z=\pd[[E_p]]$ or $Z$ is the unique minimizing hypersurface in $E_p$ Since $s_j$ in the blow-up is always at distance $1$ from $p$, in either case, we arrive at a contradiction with Allard regularity theorem.
		
		We will now go on to prove that $Z=\pd[[O]]$ is stationary one-sided minimizing. It suffices to show this for $Z$ restricted to almost every ball of large enough radius centered at $0$. The idea is to assume it's not one-sided minimizing, and then use a cut-and-paste argument on the hypothetical one-sided minimizer to deduce a contradiction. 
		
		We have proven that the associated varifold of $\pd[[O_j]]$ converges to $\mu_Z$, and thus $Z$ is already stationary. Suppose $Z$ is not one-sided minimizing in $B_r(0)$. Since $B_r(0)$ is contractible, the De Giorgi formulation (Caccioppoli sets) and Federer-Fleming formulation (integral currents) of Plateau problem is the same. Let $O_j=|s_j-p|\du E_j$. Now we solve the De Girogi formulation of Plateau problem with boundary $L$ and competitors in $O\cap B_r(0).$ To be precise, we're looking for minimizers of perimeter among Caccioppoli sets $O'$ with $O\sm B_r(0)=O'\sm B_r(0)$, and $O'\lp B_r(0)\s O\lp O$ so that $O'$. We denote the minimizer by $O'$ and call the boundary as an integral current $Z'=\pd[[O']]\lp B_r(0)$. By our assumption that $Z$ is not one-sided minimizing, we have $\ma(Z'\lp B_r(0))<=\ma(Z\lp B_r(0))$. Consider the Caccioppoli set $E_j\cup O'$ If $E_j\sm O'\cs B_r(0),$ then by definition of one-sided minimizing, $P(E_j\cap O',B_r(0))\ge P(E_j,B_r(0))$. Recall that $P(E_j\cap O',B_r(0))+P(E_j\cup O',B_r(0))\le P(E_j,B_r(0))+P(O',B_r(0))$ (Proposition 3.38 (d) in \cite{LA}), so we deduce that $P(E_j\cup O',B_r(0))\le P(O',B_r(0)).$ If $E_j\sm O'$ is not compactly contained in $B_r(0)$, we can still get an inequality of the form
		\begin{equation}\label{kin}
		P(E_j\cap O',B_{r}(0))\le P(E_j,B_{r}(0))+\e_j
		\end{equation}
		where $\e_j$ is $O(d_H(\spt Z\cap _r(0),\spt Z_j\cap B_r(0)),$ and $\de>0$ a small number with $d_H$ being Hausdorff distance. This can be seen as follows. The union $U$ of all relative non-compact connected components of $E_j\sm O'$ is also a Caccioppoli set. It must have boundary intersecting $B_r(0)$ nontrivially. Moreover, in a slightly larger ball $B_{r+\de}(0),$ $P(O',B_{r+\de})$ and $P(O'+U,B_{r+\de})$ differs only by the measure of the part of $\pd U$ lying on $\pd B_r(0)$. Since $\pd U\cap B_r(0)$ is always contained in $\spt E_j\De O'\lp B_r(0),$ (which is defined for $r$ a.e. by slicing theory), from Hausdorff convergence of support of $Z_j$ to $Z$, the part of $\pd U$ lying on $\pd B_r(0)$ always have $n$-dimensional Hausdorff measure of $O(d_H(\spt Z\cap B_r(0),\spt Z_j\cap B_r(0))$. Thus, we can use $O'+U$ to replace $O'$ and use the one-sided minimizing property of $Z_j$ to deduce \ref{kin}. Thus, we always have
		\begin{align*}
		P(E_j\cup O',B_r(0))\le P(O',B_r(0))+O(d_H(\spt Z\cap B_r(0),\spt Z_j\cap B_r(0))
		\end{align*}
		Let $j\to\infty,$ and take a subsequence if necessary. We have $E_j\cup O'\to E\cup O=E$ as Caccioppoli set, because $1_{E_j\cup O'}\to 1_{E\cup O}$ in $L^1.$ By lower-semi continuity of perimeter, we deduce that 
		\begin{align*}
		P(E,B_r(0))\le \liminf_j P(E_j\cup O',B_r(0))\le P(O',B_r(0))<P(E,B_r(0)),
		\end{align*}
		which is the desired contradiction.
	\end{proof}
	\section*{Acknowledgement}
	This paper originated from an email exchange with Professor Fanghua Lin. The author is very grateful to his support. The only new conclusions in this paper not studied by him before are the ones whose proof require a careful analysis of current convergence compared to the associated varifold convergence. The author is also grateful for an email exchange with Professor Leon Simon about counterexamples regarding the two kinds of convergence.

\end{document}